\documentclass[fleqn]{amsart}
\usepackage[utf8]{inputenc}
\usepackage[american]{babel}
\usepackage{zref-clever}
\zcsetup{cap=true}

\input{000_setup}

\title[]{Conic locus of inversive Poncelet circumcenter\\and two points of invariant circle power}

\author[R. Garcia]{Ronaldo A. Garcia$^a$}
\thanks{$^a$Instituto de Matemática e Estatística, Universidade Federal de Goiás, Goiânia, Brazil. \texttt{ragarcia@ufg.br}}

\author[S. Helman]{Shmuel (Mark) Helman$^b$}
\thanks{$^b$\texttt{helmanshmuel1@gmail.com}}

\author[D. Reznik]{Dan Reznik$^c$}
\thanks{$^c$Data Science Consulting, Rio de Janeiro, Brazil. \texttt{dreznik@gmail.com}}

\begin{document}

\begin{abstract}
We prove that over a generic Poncelet triangle family, the locus of the circumcenter of an inversive triangle is a conic. Additionally, we prove an earlier conjecture: over generic Poncelet triangles, two unique points exist which maintain constant power with respect to the circumcircle and Euler's circle of the family, respectively.
\end{abstract}

\maketitle


\section{Introduction}
\label{sec:intro}
We consider a Poncelet family of triangles $ABC$ \cite{dragovic11} inscribed in a first ellipse $\E$ and circumscribing a second, nested one called $\E_c$, \zcref{fig:poncelet}. Referring to \zcref{fig:main}, we add to this setup an \ti{inversion circle} $\mathcal{K}$ centered at $O$ and consider the family of \ti{inversive triangles} $A'B'C'$, whose vertices are inversions of $A,B,C$ with respect to $\mathcal{K}$. While it is known that the circumcenter $X_3$ (using Kimberling's notation \cite{etc}) of $ABC$ sweeps a conic \cite{helman2021-power-loci,helman2021-theory}, we prove that the circumcenter $X_3'$ of the inversive triangle -- distinct frm the inversive image of $X_3$ with respect to $\mathcal{K}$, see \zcref{fig:similitude} -- also sweeps one, thanks to its algebraic format, see \zcref{lem:complex-projective}. This is unexpected since the inversive family is not Ponceletian (it is inscribed in a degree-four curve). Indeed, experiments suggest that no other center of the inversive triangle sweeps a conic locus, see \zcref{fig:xk-nonconic}. We show that the type of conic swept by $X_3'$ depends on the location of $O$ with respect to the region swept by the circumcircle, see \zcref{fig:exterior,fig:interior}. We also show that $O$ is the internal similitude center -- where the two `internal' tangents meet \cite{mw} -- between the conic loci of $X_3$ and $X_3'$, see \zcref{fig:similitude}. 

Also proved here is a related phenomenon, originally conjectured in \cite{helman2021-power-loci}: in any generic Poncelet triangle family, two points $P_3$ (always interior to $\E'$) and $P_5$ exist which hold constant power with respect to the circumcircle and Euler's circle, respectively.

\begin{figure}
\centering
\includegraphics[width=.8\linewidth]{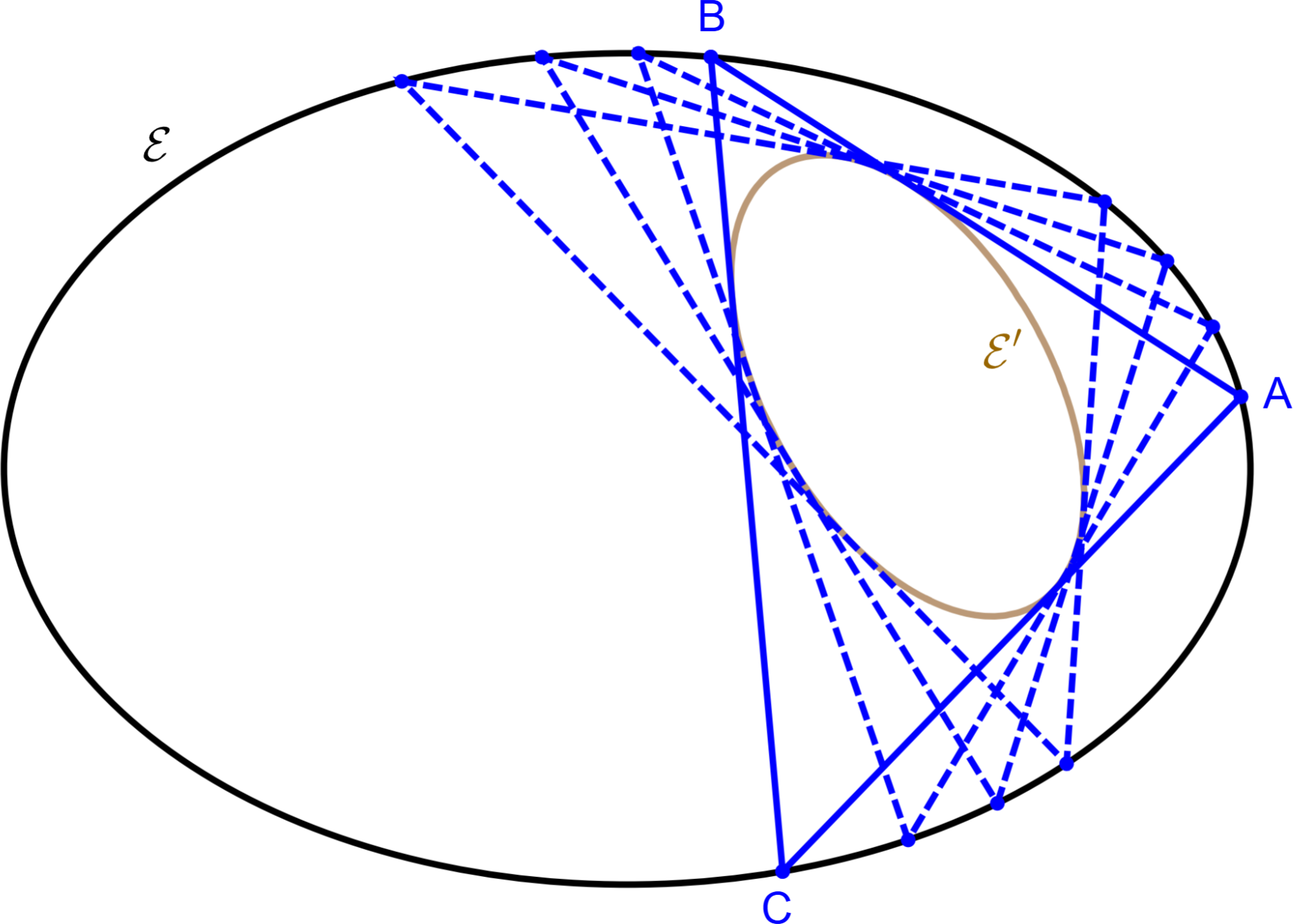}
\caption{If an n-gon (in this case a triangle $ABC$, solid blue) is inscribed in a first conic $\E$ while being simultaneously circumscribed about a second conic $\E'$, Poncelet's porism states that a one-dimensionaly family of such n-gons (dashed blue) exists which maintains said properties \cite{dragovic11}. \tb{Video}: \hrefs{https://youtu.be/UTdGwAIjuT8}}
\label{fig:poncelet}
\end{figure}

\begin{figure}
\centering
\includegraphics[width=\linewidth]{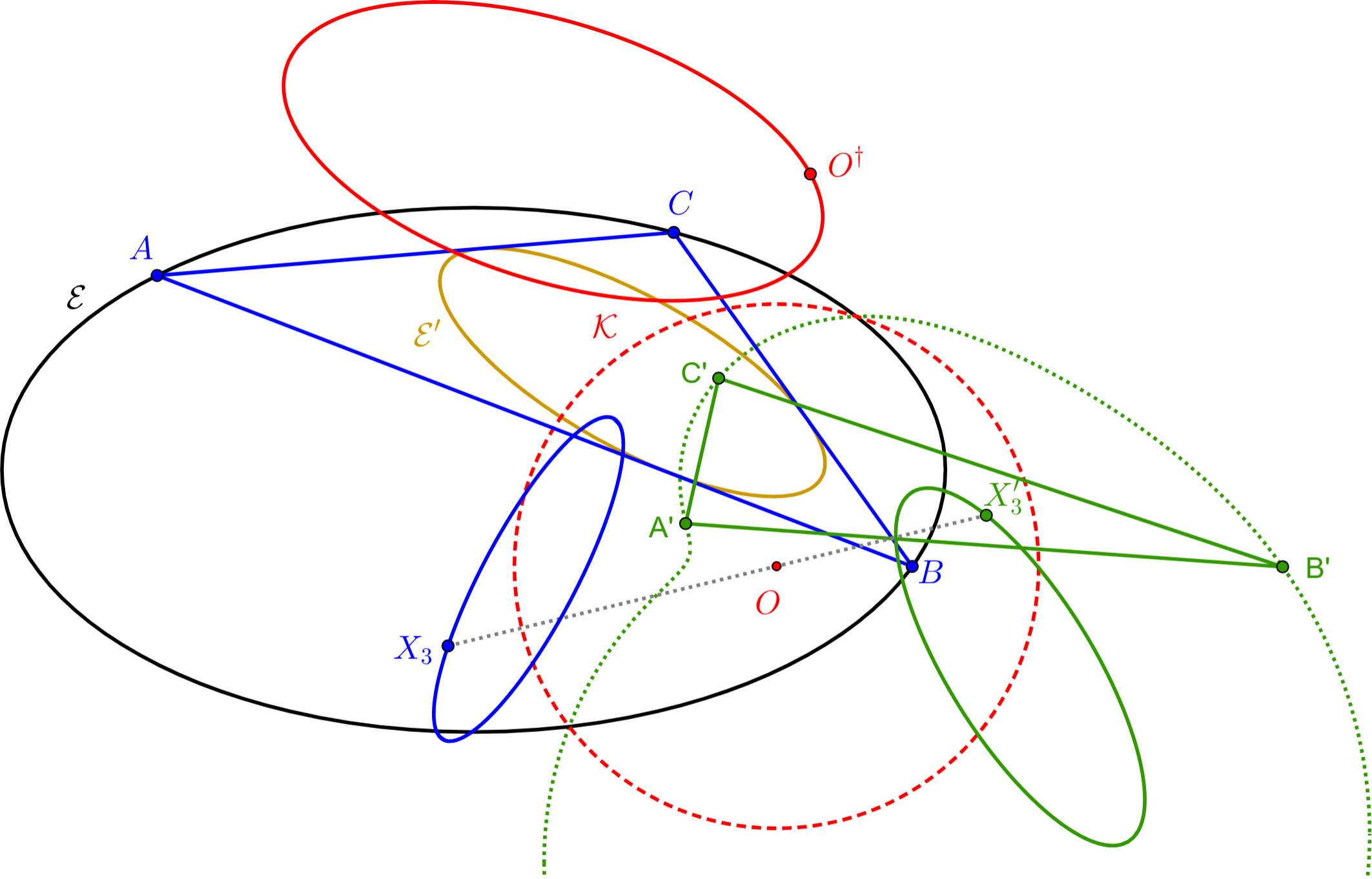}
\caption{$T=ABC$ (blue) is a Poncelet triangle interscribed between conics $\mathcal{E}$ and $\mathcal{E}'$. The locus of its circumcenter $X_3$ (blue) is guaranteed to be a conic \cite{helman2021-theory,helman2021-power-loci}. Let $O$ be a fixed point. The locus of its isogonal conjugate $O^\dagger$ is a conic (red) \cite{garcia2026-x4-conjugate}. The inversive triangle $T'=A'B'C'$ (green) has vertices at the inversions of $A,B,C$ with respect to a circle $\mathcal{K}$ (dashed red) centered on $O$. Over Poncelet its vertices sweep a quartic (dashed green, shown partially). Nevertheless, the locus of its circumcenter $X_3'$ sweeps a conic (solid green), proved in \zcref{thm:conic-inv-x3}. \tb{Video}: \hrefs{https://youtu.be/xoqDV_FbwC0}}
\label{fig:main}
\end{figure}

\begin{figure}
\centering
\includegraphics[width=\linewidth]{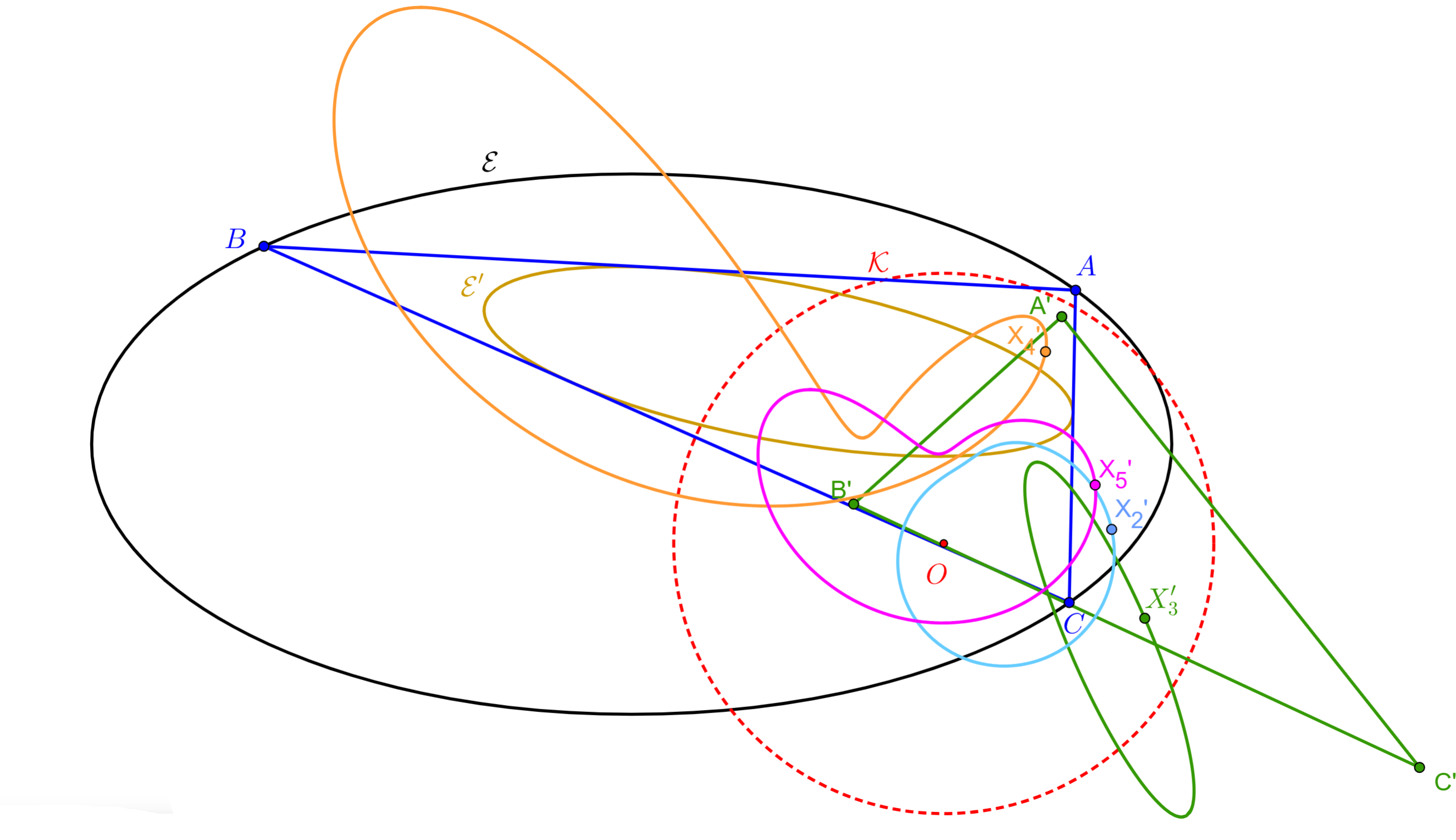}
\caption{$T=ABC$ (blue) is a Poncelet triangle interscribed in $\E,\E'$. As argued, the locus of the circumcenter $X_3'$ of the inversive triangle $T'=A'B'C'$ with respect to circle $\mathcal{K}$ (green) is a conic (solid green), the loci of the barycenter $X_2'$ (light blue), orthocenter $X_4'$ (orange), and Euler center $X_5'$ (magenta) are non-conic. Indeed, experiments suggest that no triangle center of $T'$ (other than $X_3'$) will in general sweep a conic. \tb{Video}: \hrefs{https://youtu.be/erA0irUaEjo}}
\label{fig:xk-nonconic}
\end{figure}

\subsection{Related work}

We have recently shown that over a generic Poncelet family, the locus of the isogonal conjugate of a fixed point $O$ is a conic (see the locus of $O^\dagger$ in \zcref{fig:main}), where the conic type is determined by the same conditions applicable to the locus of $X_3'$, namely the location of $O$ with respect to the region swept by the circumcircle \cite{garcia2026-x4-conjugate}.

Pioneering works on the dynamic geometry of Poncelet triangles include \cite{odehnal2011-poristic,zaslavsky2001-poncelet}, and those which followed them, e.g., \cite{dragovic2025-parable,jurkin2024-poncelet,kodrnja2023-locus,koncul2018-isot,mura2026-poncelet}. We too have explored various dynamic phenomena, e.g., \cite{garcia2024-incircle,garcia2020-new-properties,helman2021-theory,helman2021-power-loci}. 

Some of our proofs emerge from CAS-based manipulating and simplification of the long algebraic expressions stemming from the parameterization of Poncelet triangles described in \zcref{app:symmetric}. Nevertheless, the reader is encouraged to investigate applicable tools of algebraic geometry, e.g., those deployed in \cite{olga2014-incenters,schwartz2016-com,zaslavsky2001-poncelet}.

\subsection{Article organization}

In \zcref{sec:circum-inversive} we derive expressions for the circumcenter of the inversive triangle. In  \zcref{sec:locus-inversive} we prove our main result, i.e., that the locus of the circumcenter of the inversive triangle is indeed a conic. We also show that the center of inversion is the internal similitude center of the conic loci of $X_3$ and $X_3'$. Finally, in \zcref{sec:locus-power} we derive explicit expressions for the points of constant circumcircle and Euler's circle power.

In \zcref{app:symmetric} we specify a symmetric parametrization for Poncelet triangles based on Blaschke products. In \zcref{app:inversive} we derive basic properties of the inversive triangle and its circumcenter. Finally, in \zcref{app:kappas} we list the constants used in the expression of $P_5$.


\section{Circumcenter of inversive triangle}
\label{sec:circum-inversive}
Let $\A(z)=pz+q\ol{z}$ denote the linear transformation that takes the vertices $z_i$ of a triangle inscribed in the unit circle to those inscribed in the outer Poncelet ellipse $\E$. The inversion of a complex number $z$ with respect to a circle $\K(z_0,r)$ is given by $inv(z,\K)=  z_0+\frac{r^2(z-z_0)}{|z- z_0|^2}=z_0+\frac{r^2}{\ol{z}-\ol{z_0}}$. Define the \ti{inversive triangle} as having vertices $z_i'$ given by:
\[ z_i'=(inv\circ\A)(z_k), \;\;\; k=1,2,3.\]

\begin{lemma}    
Let $\{z_1,z_2,z_3\}$ be a Poncelet triangle inscribed in the unit circle and circumscribed in a ellipse with foci $f$ and $g$. The circumcenter of $\{z_1',z_2',z_3'\}$ is given by: \[ C=z_0+\frac{r^2 (a_2\lambda+a_1\ol{\lambda}+a_0)}{b_2\lambda+b_1\ol{\lambda}+b_0},\]
where:
\begin{align*}
 a_2=&-p q \left(q\ol{g}   \ol{f}-p \right) r^{2},\\
 a_1=&r^{2}( \lambda +p q \left(f g p -q \right) ),\\
 a_0=&-\left((f+g) p \,q^{2}-p^{2} q ( \ol{f} +\ol{g})  +(p^{2}   -q^{2})  {z_0} \right) r^{2},\\
 b_2=&\ol{b_1}=-q p \left(\ol{f} \ol{g}(p z_0- q \ol{z_0}) +(q^2-p^2)(\ol{f}   +\ol{g} )+p \ol{z_0}- q{z_0}  \right),\\
 b_0=& -p^{4}+q \left(f g +\ol{f} \ol{g}\right) p^{3}-
 \left(z_0 (f+g)+\ol{z_0} (\ol{f} +\ol{g})\right) p^{2} q + |z_0|^2\,p^{2}-\left( f g +\ol{f} \ol{g}\right) q^{3} p +q^4,\\
 +&\left(\ol{z_0} (f +g)+ z_0(\ol{f}+  \ol{g})\right) p \,q^{2} -  |z_0|^2 \,q^{2}.
\end{align*} 
\label{prop:circumcenter_inversive triangle}
\end{lemma}
\begin{proof}
Consider the family of triangles $\{z_1,z_2,z_3\} $ parametrized as described in \zcref{SymPar}. It follows that:
\begin{align*}
z_k'=&\frac{r^{2} z_k}{q \,z_k^{2}-\ol{z_0}z_k +p}, \quad k=1,2,3.
\end{align*}
The circumcenter $X_3$ of a triangle $\{w_1,w_2,w_3\}$ is given by \cite[Lemma 6.24, p.108]{chen2024-euclidean}:
\[X_3 = \frac{|w_1|^2(w_2 - w_3) + |w_2|^2(w_3 - w_1) + |w_3|^2(w_1 - w_2)}{\ol{w}_1(w_2 - w_3) + \ol{w}_2(w_3 - w_1) + \ol{w}_3(w_1 - w_2)}\]

Injecting the above into the symmetric parametrization of \zcref{SymPar} yields aand simplifying the expression above applied to the triangle $\{z_1',z_2',z_3'\}$ yields the claim.
\end{proof}

\section{Locus of the inversive circumcenter}
\label{sec:locus-inversive}
Let $\E$ and $\E_c$ be two nested ellipses which admit a family $\T$ of Poncelet triangles, i.e., they satisfy Cayley's condition for $n=3$ \cite{dragovic11}. Let $T=ABC$ be an element of $\T$ and $T'=A'B'C'$ have vertices which are inversions of $A,B,C$ with respect to a fixed circle $\K=(O,r)$, \zcref{fig:inversive}. In \zcref{app:inversive} we show that the inversive circumcenter $X_3'$ is collinear with $X_3$ and $O$. The following Lemma will be instrumental in proving our main result:

\begin{lemma}
Let $f:\Cp\rightarrow\Cp$ be a function given by:
\[f(z)=\frac{k_1 z+k_2\ol z+k_3 }{k_4 z+k_5\ol z+k_6},\]
with $k_j\in\Cp$, $1\leq{j}\leq{6}$. Suppose $k_4=\ol k_5$ and $k_6\in\R$. Identify $\mathbb{C}$ with the affine part of $\mathbb{RP}^2$ by sending $x+iy$ to $[x:y:1]$. Then $f$ is a projective transformation and thus preserves conics.
\label{lem:complex-projective}
\end{lemma}
\begin{proof}
Write $k_j=u_j+i\, v_j$ with $u_k,v_k\in\R$ for each $j=1,2,3$. Since $k_4=\ol k_5$ and $k_6\in\R$, we can define $w_1:=k_4+k_5\in\R$, $w_2:=(k_4-k_5)/i\in\R$, and $w_3:=k_6$. Now write $z=x+i\,y$ with $x,y\in\R$. Identifying $\Cp$ with $\R^2$, we can rewrite $f$ as $f_{real}:\R^2\rightarrow \R^2$ as:
\[f_{real}(x,y)=\left( \frac{u_1 x+u_2 y+u_3 }{w_1 x+w_2 y+w_3}, \frac{v_1 x+v_2 y+v_3 }{w_1 x+w_2 y+w_3}\right)\]
Now, identify each point $(x,y)\in\R^2$ with $(x,y,1)$ in the projective plane $ \mathbb{RP}^2$, we can rewrite $f_{real}$ as $f_{proj}:  \mathbb{RP}^2\rightarrow \mathbb{RP}^2$ with:
\[f_{proj}(x,y,z)=\left(u_1 x+u_2 y+u_3 z,v_1 x+v_2 y+v_3 z,w_1 x+w_2 y+w_3 z\right)\]
This is a linear transformation on $\mathbb{RP}^2$, also known as a projective transformation. As is known, these transform conics to conics, as desired.
\end{proof}

Let $\L_3'$ denote the locus of the circumcenter of $T'$ over $\T$. 
Referring to \zcref{fig:main}, our main result follows directly from \zcref{prop:circumcenter_inversive triangle}
and \zcref{lem:complex-projective}, namely:

\begin{theorem}
\label{thm:conic-inv-x3}
$\L_3'$ is a conic.
\end{theorem}


Referring to \zcref{fig:xk-nonconic}, experiments suggest:

\begin{conjecture}
In a generic Poncelet family, the only triangle center of $T'$ which sweeps a conic is the circumcenter $X_3'$.
\end{conjecture}

Let $\Rm$ be the region swept by the circumcircle of a generic Poncelet triangle family. It has been shown that $\Rm$ is delimited by two closed curves given implicitly by a polynomial of degree 4 \cite[Prop 2]{garcia2026-x4-conjugate}. Referring to \zcref{fig:exterior,fig:interior}:

\begin{figure}
\centering
\includegraphics[width=1\linewidth]{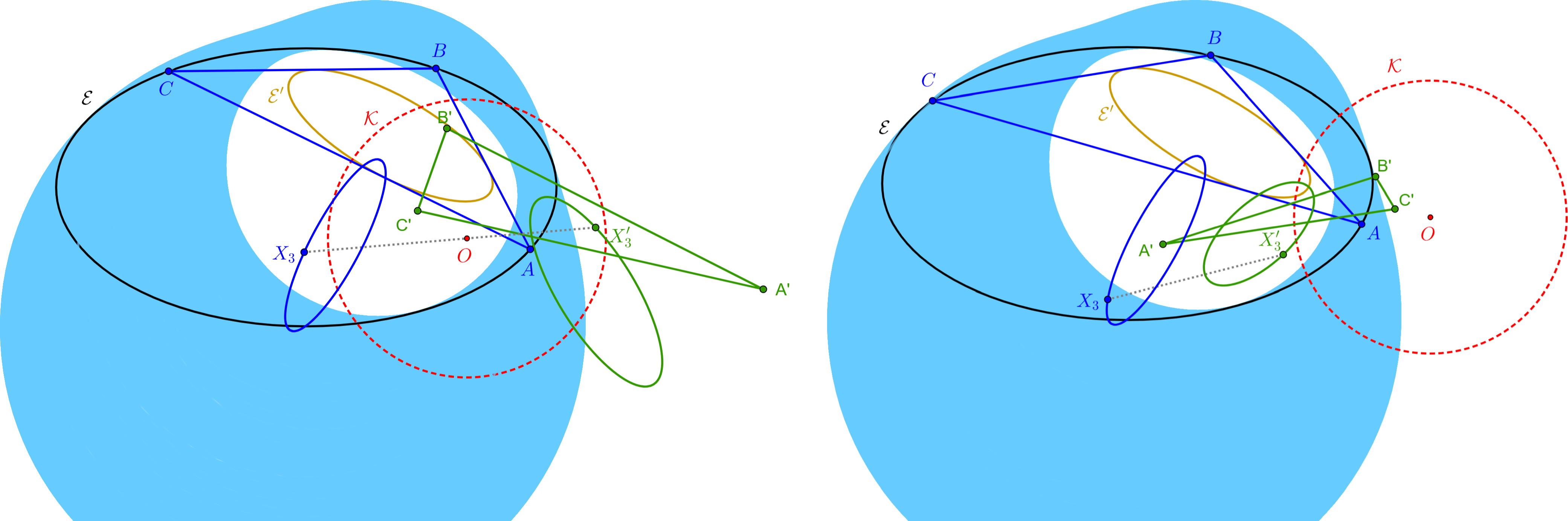}
\caption{Two cases (left and right) where the center of inversion $O$ is exterior to the region (light blue) swept by the circumcircle of Poncelet triangles $ABC$. In such cases, $\L_3'$ (green) is an ellipse.}
\label{fig:exterior}
\end{figure}

\begin{figure}
\centering
\includegraphics[width=.6\linewidth]{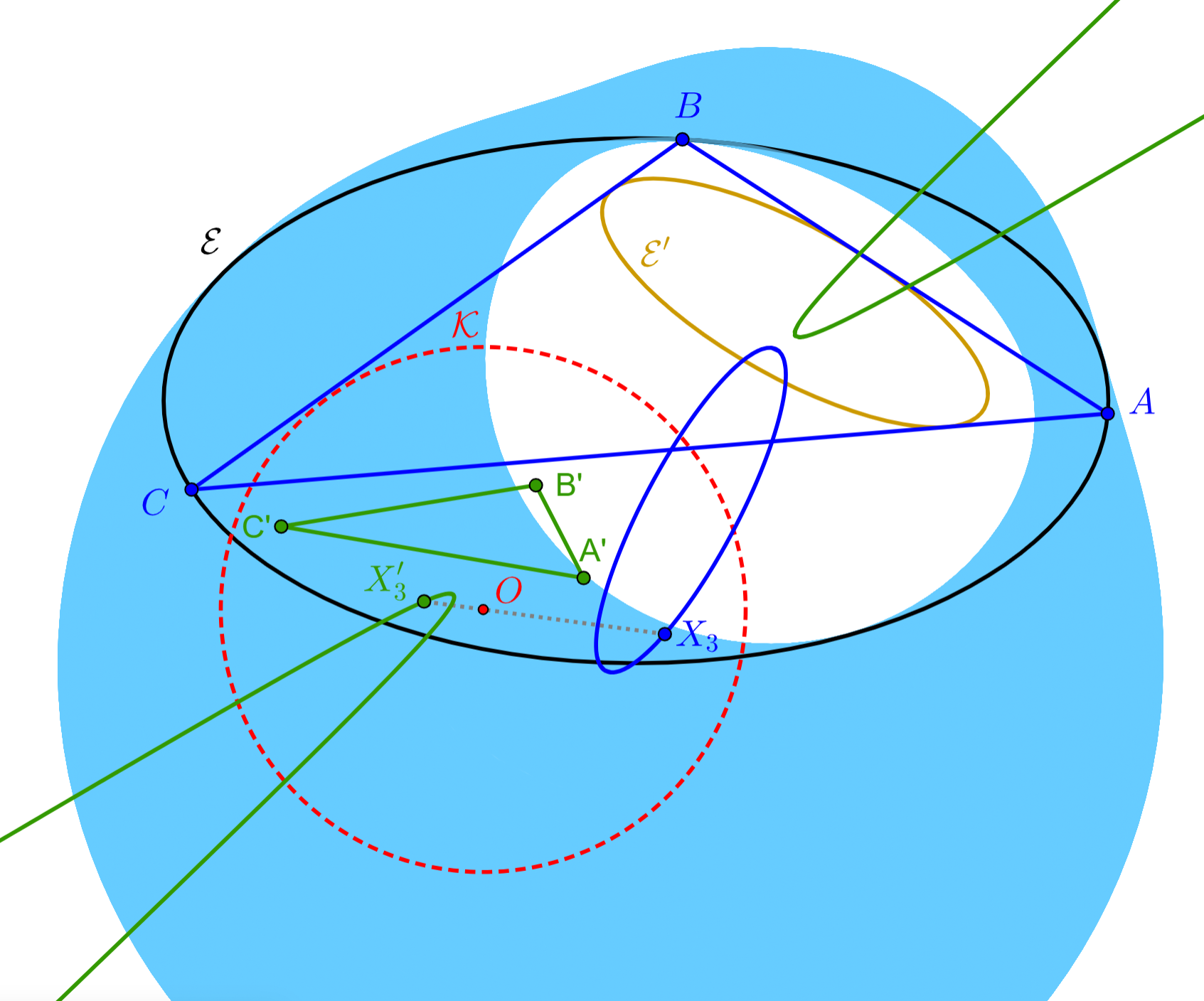}
\caption{When $O$ is interior to the region (light blue) swept by the circumcircle, $\L_3'$ (green) is a hyperbola.}
\label{fig:interior}
\end{figure}

\begin{proposition}
If $O$ is exterior (resp. interior) to $\Rm$, $\L_3'$ is an ellipse (resp. hyperbola). If $O$ is on $\partial\Rm$, $\L_3'$ is a parabola.
\label{prop:exterior}
\end{proposition}

\begin{proof}
When $O$ is exterior (resp. interior) to $\Rm$, it is never crossed (resp. six-times crossed) by the moving circumcircle. When $O$ is on the circumcircle, the vertices of $T'$ become collinear (inversion will send a circle to a line), and its circumcenter is pushed to the line at infinity. Whe $O$ is on $\partial\Rm$, it is 3-times crossed (3 double roots), yielding a parabola.
\end{proof}

\begin{definition}
Given two circles, the internal (resp. external) similitude center is the intersection of their two common internal (resp. external) tangents. 
\end{definition}

Below we use this concept with respect to the two common internal tangents between two conics, since there is exists a projectivity that sends them to two circles.

Let $\L_3$ refer to the locus of the circumcenter of the generic Poncelet family $\T$. This is guaranteed to be a conic \cite[Thm.2]{helman2021-power-loci}. Referring to \zcref{fig:similitude}, we thank Dominique Laurain \cite{laurain2026-similitude} for observing that:

\begin{figure}
\centering  \includegraphics[width=\linewidth]{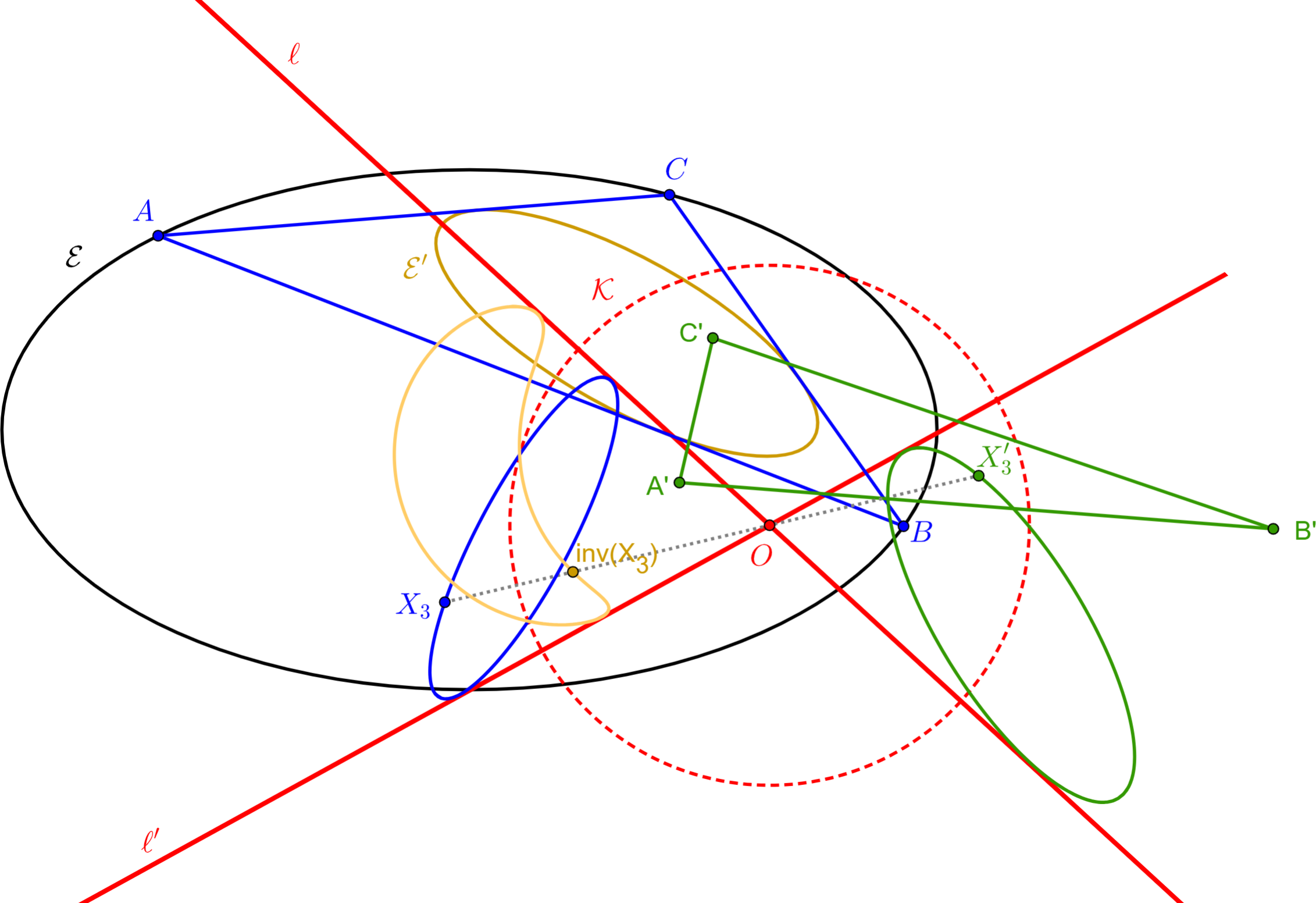}
\caption{Lines $\ell$ and $\ell'$ (both red) are `internally' tangent to $\L_3$ (blue) and $\L_3'$ (green). Surprisingly, these meet at $O$, i.e., the latter is the internal similitude center of said two conic loci, see \zcref{prop:similitude}. Also shown is the non-conic locus of the inversion of $X_3$ with respect to $\mathcal{K}$ (yellow), labeled $\text{inv}(X_3)$. Notice how $\ell,\ell'$ are also tangent to it.}
\label{fig:similitude}
\end{figure}

\begin{proposition}
$O$ is the internal similitude center of $\L_3$ and $\L_3'$.
\label{prop:similitude}
\end{proposition}

\begin{proof}
In \zcref{app:inversive} we show that $X_3$, $O$, and $X_3'$ are collinear. As $O$ is fixed,  the map
$\sigma(X_3)=X_3'$ is a diffeomorphism between the two conics. Therefore, $OX_3$ is tangent to the locus of $X_3$ if and only if $OX_3'$ is tangent to the locus of $X_3'$.

\end{proof}

Let $\text{inv}X_3$ denote the inversion of $X_3$ with respect to $\mathcal{K}$. Still referring to \zcref{fig:similitude}, it can be shown that:

\begin{observation}
The internal tangents between $\L_3$ and $\L_3'$ are also tangent to the locus of $\text{inv}X_3$. 
\end{observation}

\begin{proof}
Direct also from the fact that the map
$\tau(X_3)= \text{inv}X_3$ is a diffeomorphism.
\end{proof}

This phenomenon is further illustrated in \zcref{fig:diffeo}.


\begin{figure}
\centering
\includegraphics[width=\linewidth]{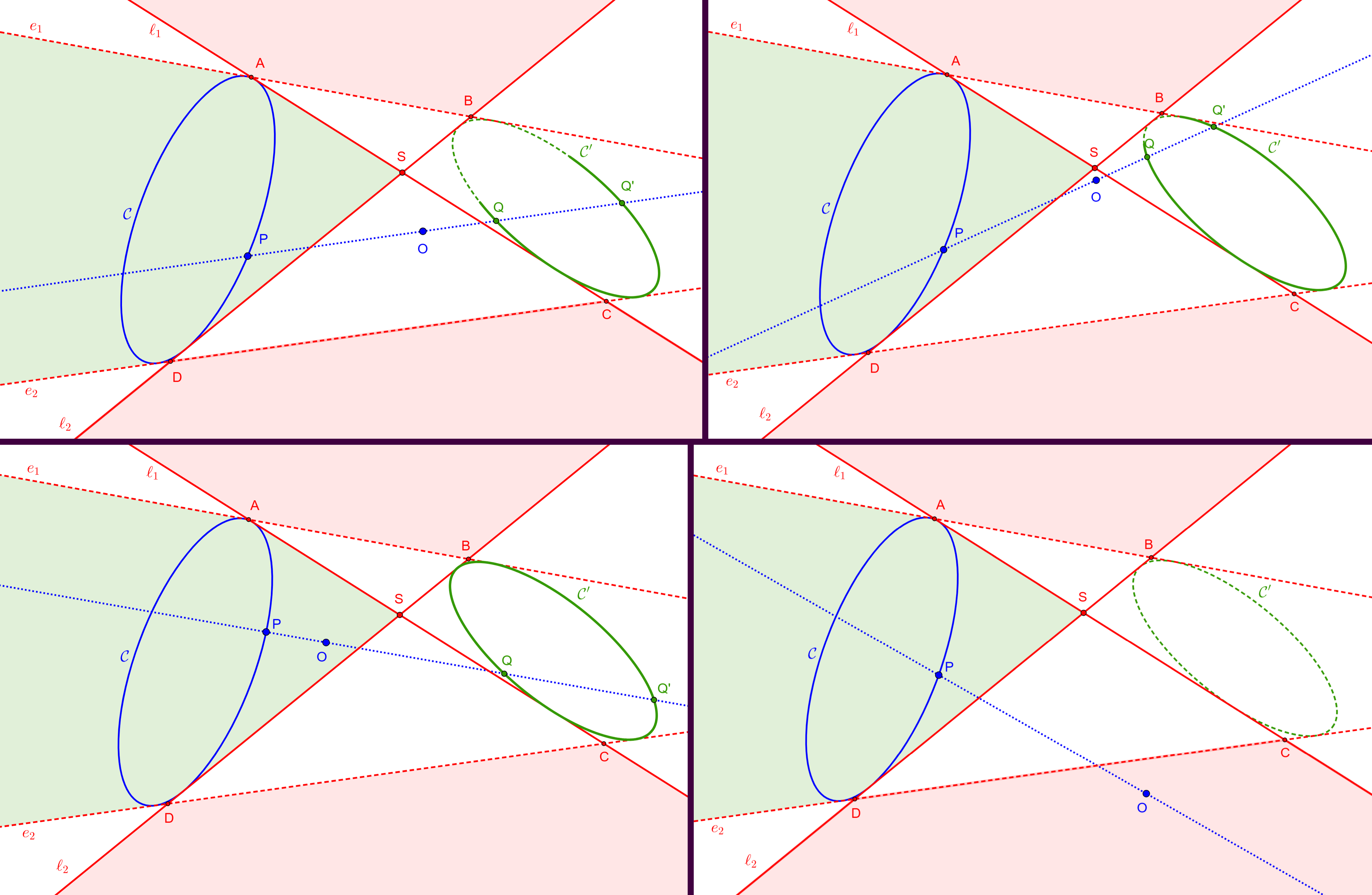}
\caption{In all four figures, $\mathcal{C},\mathcal{C}'$ are ellipses (blue, dashed green), and $\ell_1,\ell_2$ (resp. $\mathrm{e}_1,\mathrm{e}_2$) are common `internal' (resp. external) tangents, drawn in red (resp. dashed red). The former meet at their internal center of similitude $S$. Their other pairwise intersections are labeled $A,B,C,D$. $P$ is a point on $\mathcal{C}$, and $O$ is a fixed point. $Q$ and $Q'$ are the  intersections of $PO$ with $\mathcal{C}'$. \tb{top left}: $O$ is the $SDC$ triangle. Over a revolution of $P$ on $\mathcal{C}$, $Q,Q'$ only partially sweep $\mathcal{C}'$ (solid green); \tb{top right}: with $O$ closer to $S$ and still in $SDC$, achieves a nearly complete cover of $\mathcal{C}$ (missing piece is dashed green); \tb{bottom left}: $O$ is in the $DSA$ unbounded polygon (light green), and all of $\mathcal{C}'$ is swept; the same is achieved if $O$ is within another unbounded polygon region to the `right' of the external similitude center $S'=\mathrm{e}_1\cap\mathrm{e}_2$ (not shown) and between $\mathrm{e}_1,\mathrm{e}_2$; \tb{bottom right}: when $O$ is in either open polygons $DC$ or $AB$ (light red), none of $\mathcal{C}'$ is ever crossed by $PO$, i.e., $Q,Q'$ are always complex. \tb{Video}: \hrefs{https://youtu.be/QuBjhBwO248}}
\label{fig:diffeo}
\end{figure}

\section{Points of constant circumcircle and Euler's circle power}
\label{sec:locus-power}
\begin{lemma}
The power $\Pi$ of a point $z_0$ with respect to the circumcircle of a triangle $\{z_1,z_2,z_3\}$ is given by:
\[ \Pi=|z_0|^2+ \zeta z_0+\ol{\zeta}\ol{z_0},~
\text{with}~\zeta=\frac{|z_1|^2(\ol{ z_3 } - \ol{z_2})+|z_2|^2(\ol{ z_1 } - \ol{z_3})   + |z_3|^2(\ol{ z_2 } - \ol{z_1})}{z_1 (\ol{z_2} -\ol{z_3})+z_2( \ol{z_3}-\ol{z_1})+z_3 (\ol{z_1}-  \ol{z_2})}\]
\end{lemma}
\begin{proof}
Direct computation and simplification of $|z_0-X_3|^2-|z_1-X_3|^2$. 
\end{proof}

Let $\T$ be a generic Poncelet triangle family.
The following was conjectured in \cite[Conj.1]{helman2021-power-loci}, which we are now equipped to prove:

\begin{proposition}
Over $\T$, there is some fixed point $P_3$ (resp. $P_5$) such that its power with respect to the circumcircle (resp. Euler circle) of the family is invariant. 
\end{proposition}

Let $\Pi_3,\Pi_5$ represent the invariant powers at $P_3,P_5$, respectively. Referring to \zcref{fig:p3-p5}, we derive:

\begin{align}
P_3&= \frac{\left(g |f|^2+f |g|^2-f -g \right) p}{|f|^2  |g|^2-1}+\frac{\left(|f|^2 \ol{g}+\ol{f} |g|^2-\ol{f}-\ol{g}\right) q}{|f|^2|g|^2-1}\nonumber\\
\Pi_3 &=\frac{\left(|g|^2-1\right) \left(\ol{f} g -1\right) \left(f \ol{g}-1\right) \left(|f|^2-1\right) \left(pq( f g  +\ol{fg}  )  -p^{2}-q^{2}\right)}{\left(|f|^2 |g|^2 -1\right)^{2}}\nonumber\\
P_5&=\frac{\left(f g \left(f +g \right) p^{2} \left(\ol{fg} \;  p -q \right)+\ol{fg}\;  q^{2} \left(\ol{f}+\ol{g}\right) \left(f g q -p \right)\right) \left(p^{2}+q^{2}\right)}{2|f|^2|g|^2  \left(p^4+q^4+p^2q^2\right)-2 \left(f g +\ol{f} \ol{g}\right)  \, p q \left(p^{2}+q^{2}\right)+2 p^{2} q^{2}}\nonumber\\
\Pi_5&= (p^2 + q^2)(|fg|^2 - 1)\frac{\gamma_1}{\gamma_2}
\label{eqn:P5}
\end{align}

The long expressions for $\gamma_1,\gamma_2$ are provided in \zcref{app:kappas}.

\begin{proof}
This follows from algebraic calculations. 
Let $\mathcal{A}(z)=pz+q \ol{z}$ be the linear map and $\{z_1,z_2,z_3\}$ be a Poncelet triangle interscribed between the unit circle and an ellipse with foci $f$ and $g$.
Compute the triangle $w_i=\mathcal{A}(z_i)$ (i=1,2,3). The power of a point $w_0$ with respect to the circumcircle of $\{w_1,w_2,w_3\}$ is given by: \[ \Pi=|w_0|^2+ \zeta w_0+\ol{\zeta}\ol{w_0},~
\text{with}~\zeta=\frac{|w_1|^2(\ol{ w_3 } - \ol{w_2})+|w_2|^2(\ol{ w_1 } - \ol{w_3})   + |w_3|^2(\ol{ w_2 } - \ol{w_1})}{w_1 (\ol{w_2} -\ol{w_3})+w_2( \ol{w_3}-\ol{w_1})+w_3 (\ol{w_1}-  \ol{w_2})}.\]

Using $w_i=pz_i+q/z_i$, $\ol{w_i}=p/z_i+qz_i$ and the symmetric parametrization in \zcref{app:symmetric}, it follows that
the power $\Pi_3$ at $w_0$ is given by:
 
\begin{align*}   
\Pi_3=& \;M_1\lambda+M_2\ol{\lambda}+M_3\\
 M_1( p^{2}-  q^{2})=&   -pq[(\ol{f}\ol{g}p - q)w_0+ (-\ol{f}\ol{g}  q + p)\ol{w_0} - (p^2 - q^2) (\ol{f}+\ol{g}  )],
\\
M_2=&\;\ol{M_1},\\
M_3(p^2-q^2)=&\; -p q \left(p f -\ol{f} q +p g -\ol{g}q \right) {w_0} +\left(p^2-q^2\right)   |w_0|^2 +p q \left(q f -\ol{f} p +g q -\ol{g}p \right) \ol{w_0} \\
&+\left(p^2 -q^2 \right)   \left(f g p q +\ol{f} \ol{g}p q -p^{2}-q^{2}\right)
\\
\end{align*}
$\Pi_3$ is independent of $\lambda$ if and only if
$\left(-\ol{f} \ol{g} p +q \right) w_0 +\left(\ol{f} \ol{g} q -p \right) \ol{w_0} +\left(p^2 -q^2 \right)   \left(\ol{f} + \ol{g}\right)=0$. $P_3$ is obtained by
solving this equation. $P_5$ is obtained similarly, by considering the medial triangle of $\{w_1,w_2,w_3\}$, whose circumcircle is Euler's circle. 
\end{proof}
Referring to \zcref{fig:x56}, we note that if $\E'$ is a circle (Chapple's porism), i.e., $f=g$,  $P_3=2f/(1+|f|^2)$ which coincides with the (complex) external similitude center of $\E$ and $\E'$, or $X_{56}$ on \cite{etc}. Indeed, it is stationary over the family. It turns out that the center of $\E'$ holds constant power with respect to Euler's circle, i.e., it is $P_5$.

\begin{figure}
\centering
\includegraphics[width=0.5\linewidth]{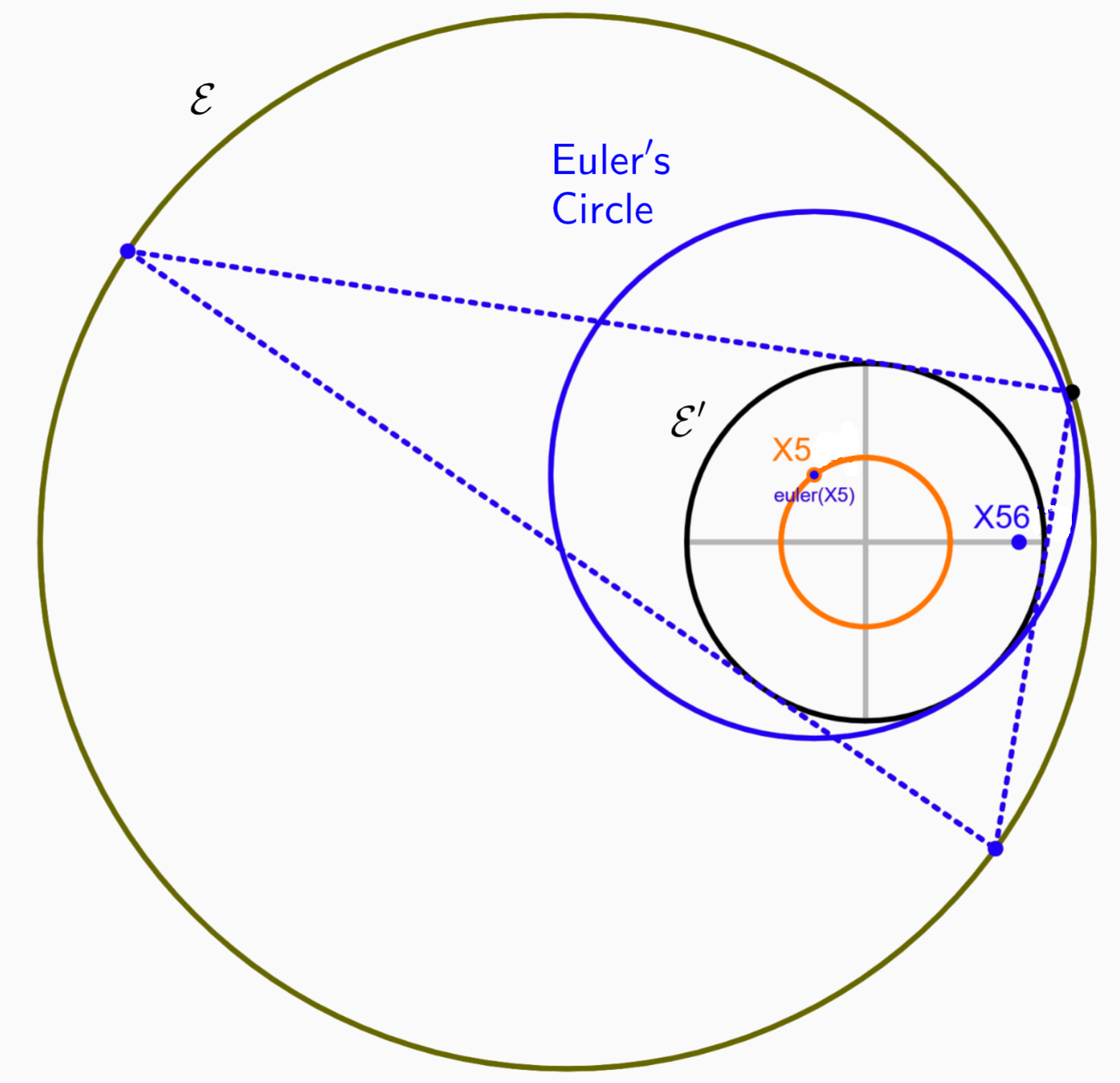}
\caption{When $\E$ and $\E'$ are circles ($f=g$ at the center of $\E'$), you get Chapple's porism (dashed blue triangles), for which $X_{56}$ is stationary and coincides with $P_3$. Since $\E$ is the circumcircle, Euler's circle (solid blue) has constant radius at half the circumradius. It is tangent to $\E'$. Its center $X_5$ moves along a circle (solid orange) concentric with $\E'$, and there it maintains constant power with respect to Euler's circle. \tb{Live}: \hrefs{https://bit.ly/4tatw7s}}
\label{fig:x56}
\end{figure}

Experimentally, we observed that:

\begin{proposition}
$P_3$ is always interior to $\E'$.
\end{proposition} 

\begin{proof}
 $P_3$  is the image by the map
$\mathcal{A}(z)=pz+q\ol{z}$ of the point $p_3=\dfrac{f+g-\left(\ol{f}+\ol{g} \right)f g}{1-|fg|^2}$. So, this is equivalent to verifying that $p_3$ is in the interior of the ellipse $|z-f|+|z-g|=|1-\ol{f}g|$. Direct calculations show that:
\begin{align*}
|z - f| = &\frac{1 - |f|^2}{1 - |f|^2|g|^2} |g|\,|1 - \bar{f}g|,\quad
|z - g| =  \frac{1 - |g|^2}{1 - |f|^2|g|^2} |f|\,|1 - f\bar{g}|.
\end{align*}
Therefore, as $|1-\ol{f}g|=|1-f\ol{g}|$, it follows that:
\[|z - f| + |z - g| = |1 - \bar{f}g| \;\left( \frac{|f| + |g|}{1 + |f||g|} \right). \]
Since $|f|<1$  and $|g|<1$ it follows that
$\frac{|f| + |g|}{1 + |f||g|} <1$. In fact, this inequality is equivalent to $(1-|f|)(1-|g|)>0$.
\end{proof}

Returning to our original problem, namely which involves a family of inversive triangles, consider the special case when the center of inversion coincides with $P_3$. Referring to \zcref{fig:O-on-p3}:

\begin{corollary}
$\L_3'$ is  homothetic (translated and scaled) to $\L_3$.
\end{corollary}

\begin{proof}
$X_3'$ is collinear with $X_3$ and $O$. $|O X_3|/|O X_3'|$ is proportional to $\Pi_3$ (\zcref{lem:ratio-ox3}), and the latter is constant.
\end{proof}

\begin{figure}
\centering
\includegraphics[width=\linewidth]{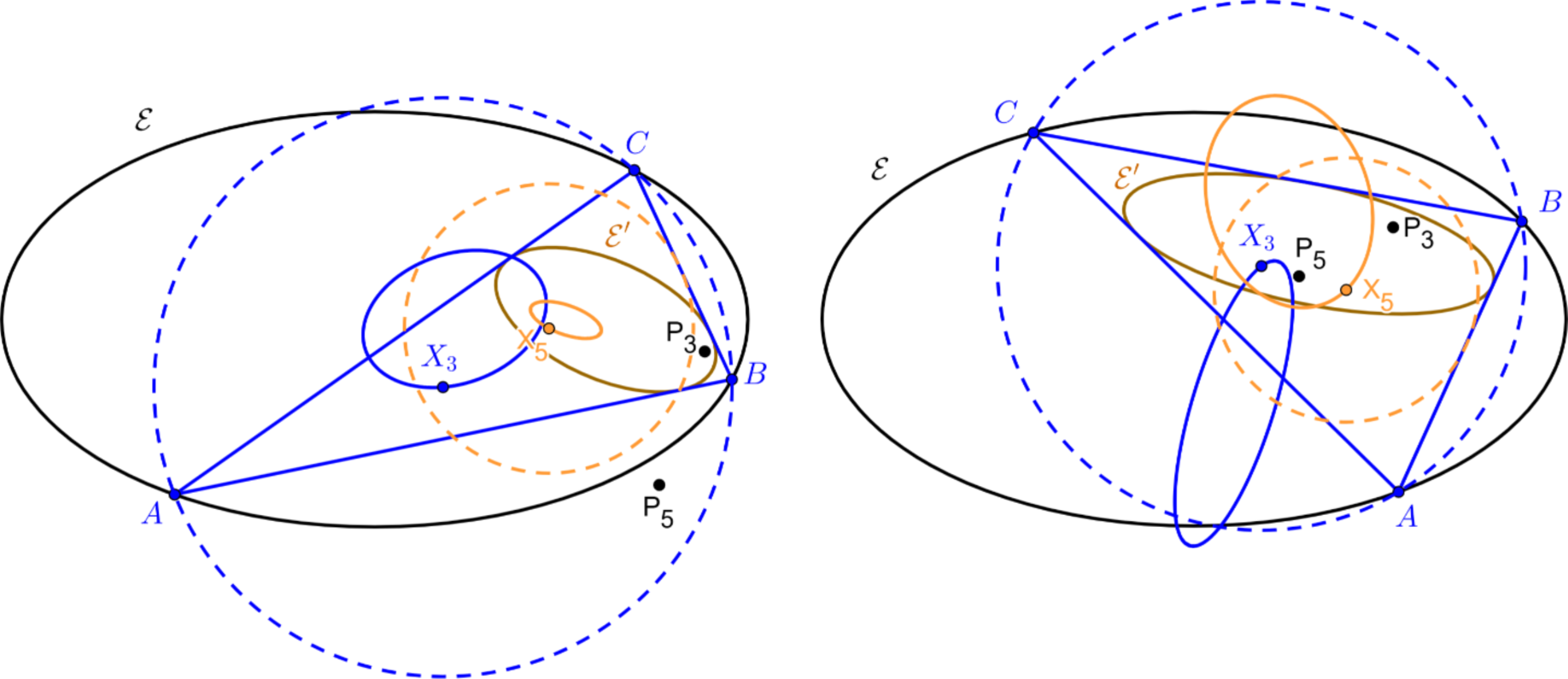}
\caption{Points $P_3$ and $P_5$ have constant power with respect to the circumcircle (dashed blue) and Euler's circle (dashed orange), respectively, for two generic Poncelet triangle families (left and right). Both are inscribed in the same outer $\E$, but circumscribe distinct inner $\E'$. Note that in both cases, $P_3$ is interior to $\E'$. Also shown are the elliptic loci of the circumcenter $X_3$ (blue) and Euler center $X_5$ (orange). \tb{Video}: \hrefs{https://youtu.be/6O6KC2bmJxw}}
\label{fig:p3-p5}
 \end{figure}

\begin{figure}
\centering
\includegraphics[width=0.8\linewidth]{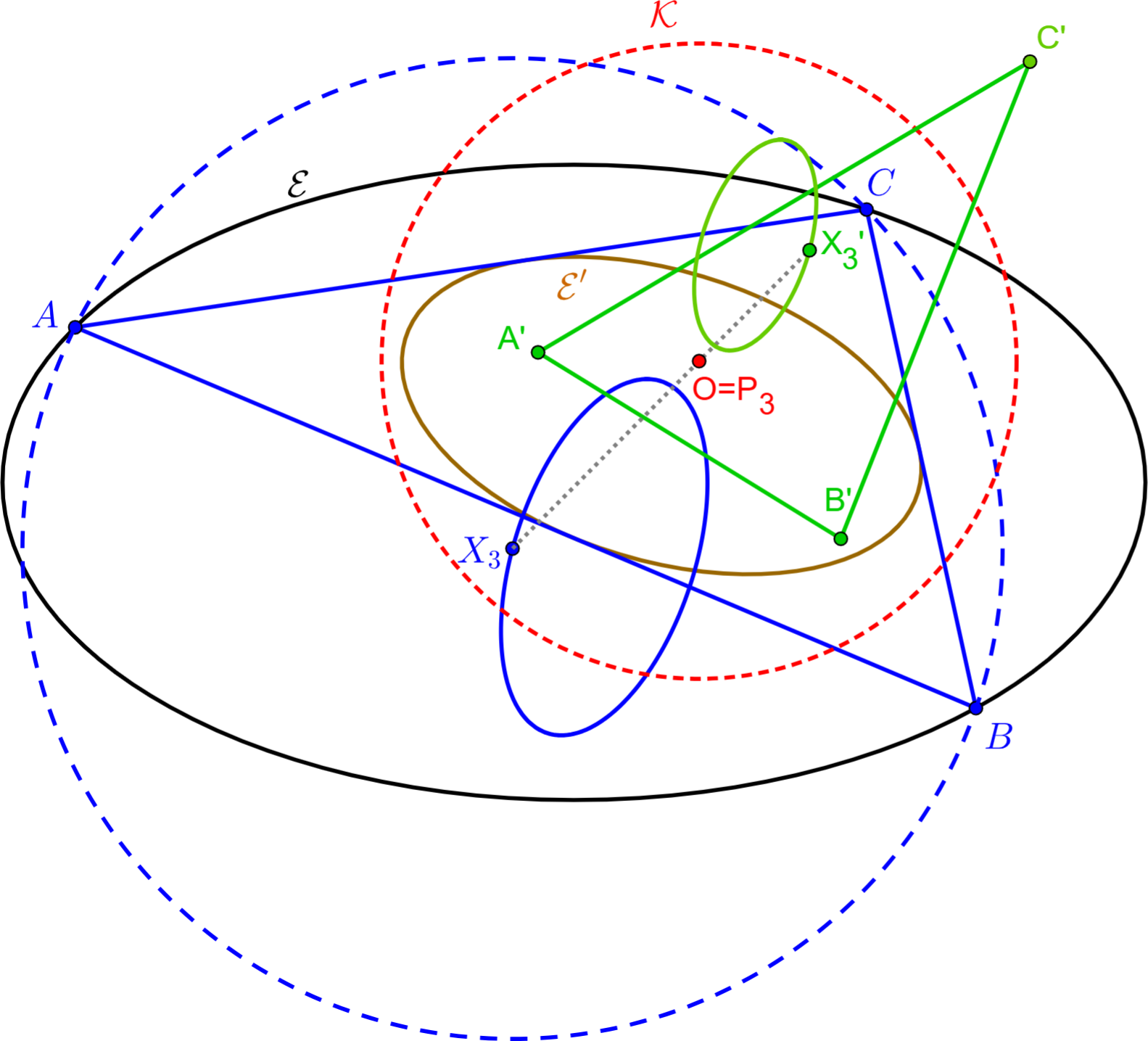}
\caption{The center $O$ of the inversion circle $\mathcal{K}$ (dashed red) is made to coincide with the point $P_3$ of constant power with respect to the circumcircle (dashed blue). In such a case, the locus of $X_3'$ (green) is a scaled/translated image of the locus of $X_3$ (blue).}
\label{fig:O-on-p3}
\end{figure}

\section*{Acknowledgements}
\noindent We would like to thank 
A. Akopyan, D. Laurain, R. Schwartz and S. Tabachnikov for their comments and encouragement. The first author is a fellow of CNPq.

\appendix

\section{Symmetric parametrization}
\label{app:symmetric}
Identifying $\mathbb{R}^2$ with $\mathbb{C}$, consider the following parameterization for Poncelet triangles inscribed in $\mathbb{T}$, the unit circle centered at the origin, as derived in \cite[Def. 3]{helman2021-power-loci} and based on the work in \cite{daepp2019-ellipses} on Blaschke products:
\begin{theorem}
For any Poncelet family of triangles inscribed in the unit circle $\mathbb{T}$ and circumscribing a nested ellipse with foci $f,g\in\mathbb{D}$ (the unit disk), parametrize its vertices $z_1,z_2,z_3\in\mathbb{T}$ as the following elementary symmetric polynomials:
\begin{align*}
z_1+z_2+z_3=& f+g+\l\ol f \ol g , \\
z_1 z_2+z_2 z_3+z_3 z_1=& f g+\l(\ol f+\ol g),\\
z_1 z_2 z_3=& \l,
\end{align*}
where the free parameter $\l=e^{i \theta}$, $\theta\in[0,2\pi]$.
\label{SymPar}
\end{theorem}

This is generalized to a Poncelet triangle family $\T$ interscribed between any two nested ellipses $\E,\E_c$ by applying an affine transformation that sends $\mathbb{T}$ to $\E$. Let $z_1,z_2,z_3\in\E$ be the varying vertices of the Poncelet triangles. The statements below are reproduced from \cite[Sec.2.2]{garcia2024-incircle}, where their proofs can also be found: 

\begin{theorem}
For any symmetric rational function $\F:\mathbb{C}^3\rightarrow\mathbb{C}$, the value of $\F(z_1,z_2,z_3)$ can be parameterized as a rational function of a parameter $\lambda$ on $\mathbb{T}$.
\end{theorem}

Let $a,b$ denote the semiaxis' lengths of $\E$, i.e., it satisfies $(x/a)^2+(y/b)^2=1$. Consider the affine transformation $\A(x,y)=(a x,b y)$ which sends the unit circle into $\E$. So $\A^{-1}(x,y)=(x/a,y/b)$. In the complex plane, $\A(z):=\frac{(a+b)}{2}z+\frac{(a-b)}{2}\ol{z}$. $\A^{-1}(z)=\frac{(1/a+1/b)}{2}z+\frac{(1/a-1/b)}{2}\ol{z}$. Let $c^2=a^2-b^2$. 

\begin{lemma}
When the inner ellipse $\E_c$ of the Poncelet triangle family is a circle, $\E_{pre}:=\A^{-1}(\E_c)$ is an axis-aligned ellipse with major semiaxis $r/b$ and minor semiaxis $r/a$, center $\A^{-1}(C)=x_c/a+i y_c/b$, and semi-focal length $r\frac{c}{a b}$, with foci given by $x_c/a+i (y_c/b\pm r\frac{c}{a b})$.
In particular, the sum and product of the foci of $\E_{pre}$ as complex numbers are given by:
\begin{align*}
f_{pre}+g_{pre}=&\frac{2x_c}{a}+\frac{2y_c}{b}\rc\\
f_{pre}{\cdot}g_{pre}=& \frac{a^2+b^2}{c^2}+\frac{2i}{a b}\left(x_c y_c+\frac1c \sqrt{(a^4-c^2 x_c^2)(b^4-c^2 y_c^2)}\right)\cdot
\end{align*}
\label{CircleTransformLemma}
\end{lemma}

\section{Inversive triangle}
\label{app:inversive}
Let $d=|OP|$ be the distance between a point $P$ and the center of a circle $\Q$ of radius $u$.

\begin{definition}
The power $\Pi(P,\Q)$ of a point $P$ relative to circle $\Q$ is given by $d^2-u^2$. 
\end{definition}

Referring to \zcref{fig:inversive}, let $\mathcal{C}$ (center $X_3$, radius $R$) and $\mathcal{C}' $ (center $X_3'$) denote the circumcircle of $T$ and $T'$, respectively, with $\K=(O,r)$ the circle of inversion.

\begin{lemma}
$\mathcal{C}'$ is in the pencil defined by $\mathcal{C}$ and $\mathcal{K}$.
\end{lemma}

While it is sufficient to argue that inversion is an anti-conformal map, here is a more detailed proof.

\begin{proof}
Let $\mathcal{P}$ be the pencil determined by $\mathcal{K}$ and $\mathcal{C}$. Consider the pencil $\mathcal{Q}$ orthogonal to the pencil $\mathcal{P}$; that is, every circle in $\mathcal{Q}$ is orthogonal to every circle in $\mathcal{P}$.

Now take any circle $\omega \in \mathcal{Q}$. By definition, $\omega \perp \mathcal{K}$. Since inversion is an anticonformal map (it preserves local angles) and it sends circles to circles, it also preserves orthogonality of circles. The inversion fixes the two intersection points $\omega \cap \mathcal{K}$ and sends $\mathcal{K}$ to itself. So the image of $\omega$ is again a circle through those same two points orthogonal to $\mathcal{K}$. Hence, the image of $\omega$ is $\omega$ itself.

Because $\omega \in \mathcal{Q}$, we also have $\omega \perp \mathcal{C}$. Since inversion preserves orthogonality of circles, the image of $\omega$ is orthogonal to the image of $\mathcal{C}$. But the image of $\omega$ is just $\omega$, and the image of $\mathcal{C}$ is $\mathcal{C}'$. Thus
$\omega \perp \mathcal{C}'$.

We have shown that any circle $\omega$ in $\mathcal{Q}$ is orthogonal to $\mathcal{C}'$. Therefore $\mathcal{C}'$ belongs to the pencil of circles orthogonal to $\mathcal{Q}$. But the pencil orthogonal to $\mathcal{Q}$ is precisely $\mathcal{P}$, the pencil generated by $\mathcal{K}$ and $\mathcal{C}$. Hence $\mathcal{C}' \in \mathcal{P}$, as desired.


\end{proof}

Since the centers of circles in a pencil have collinear centers:

\begin{corollary}
\label{cor:collinear}
$X_3$, $O$, and $X_3'$ are collinear.
\end{corollary}

\begin{lemma} Let $\K=(O,r)$ be the circle of inversion. Then,
\label{lem:ratio-ox3}
$\frac{|O X_3|}{|O X_3'|} = |\frac{|\Pi(O,\Z)|}{r^2}=|\frac{|O-X_3|^2-R^2}{r^2}|$.
\end{lemma}

\begin{proof}
The inverse of a circle of radius $R$ with center $C_1=(x_1,y_1)$ with respect to a circle with  center $C_0=(x_0,y_0) $ and  radius $r$  is another circle with center
$(x_1',y_1')=C_0+s(C_1-C_0)=\left(
 x_0+s(x_1-x_0),y_0+s(y_1-y_0)\right)$ and radius $r'=|s|R$
 where
 $1/s=(|C_1-C_0|^2-R^2)/r^2$. 
Applying to our case, it follows that
$X_3'=O+s(X_3-O)$. Therefore, $|O-X_3|/(|O-X_3'|=|1/s|$. This ends the proof.

\end{proof}

\begin{figure}
\centering
\includegraphics[width=0.6\linewidth]{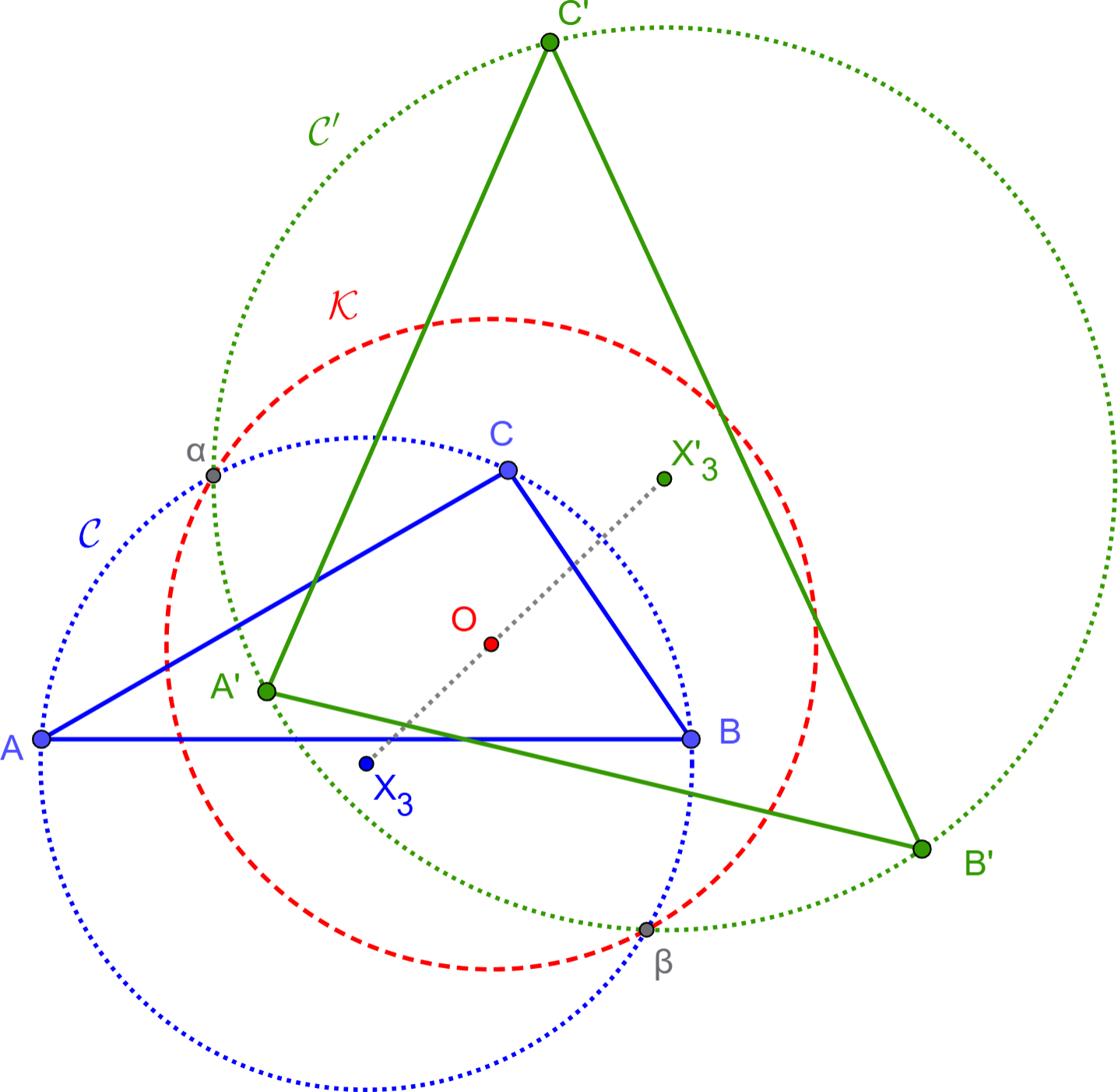}
\caption{Let $T=ABC$ (blue) be a reference a triangle with circumcircle $\mathcal{C}$ (center $X_3$) and $T'=A'B'C'$ (green), circumcircle $\mathcal{C}'$ (center $X_3'$), have vertices at the inversions of $A,B,C$ with respect to a circle $\mathcal{K}$ centered at $O$. It can be shown that (i) $\mathcal{C}'$ is in the pencil defined by $\mathcal{C}$ and $\mathcal{K}$ ($\alpha,\beta$ are common complex intersection points), (ii) $X_3'$ is collinear with $O$ and $X_3$, (iii) the ratio $|O X_3|/|O X_3'|$ is given by \zcref{lem:ratio-ox3}.}
\label{fig:inversive}
\end{figure}

\section{\torp{$\Pi_5$}{Pi5} constants}
\label{app:kappas}
The following are the constants required for \zcref{eqn:P5}.
{\small\begin{align*}
\gamma_1=&f^{2} g^{2} p^{8} {\ol{f}}^{2} {\ol{g}}^{2}-2 q f g \ol{f} \ol{g} \left(\ol{f} \ol{g}+f g \right) p^{7}\\
+&\left(2 {\ol{g}}^{2} {\ol{f}}^{2} f^{2} g^{2}-{\ol{g}}^{2} \ol{f} f^{2} g -{\ol{g}}^{2} \ol{f} f \,g^{2}-\ol{g} {\ol{f}}^{2} f^{2} g -\ol{g} {\ol{f}}^{2} f \,g^{2}+{\ol{f}}^{2} {\ol{g}}^{2}+4 f \ol{f} g \ol{g}+f^{2} g^{2}\right) p^{6} q^{2}\\
+& ( {\ol{g}}^{3} \ol{f} f g -2 {\ol{g}}^{2} {\ol{f}}^{2} f g +f g \ol{g} {\ol{f}}^{3}+\ol{g} \ol{f} f^{3} g -2 \ol{g} \ol{f} f^{2} g^{2}+\ol{g} \ol{f} f \,g^{3}+{\ol{g}}^{2} \ol{f} f +{\ol{g}}^{2} \ol{f} g \\
 + &\ol{g} {\ol{f}}^{2} f +\ol{g} {\ol{f}}^{2} g +\ol{g} f^{2} g +\ol{g} f \,g^{2}+\ol{f} f^{2} g +\ol{f} f \,g^{2}-2 \ol{f} \ol{g}-2 f g ) p^{5} q^{3}\\
+&\left(3 {\ol{g}}^{2} {\ol{f}}^{2} f^{2} g^{2}-3 {\ol{g}}^{2} \ol{f} f^{2} g -3 {\ol{g}}^{2} \ol{f} f \,g^{2}-3 \ol{g} {\ol{f}}^{2} f^{2} g -3 \ol{g} {\ol{f}}^{2} f \,g^{2}\right.\\
-&\left. {\ol{g}}^{3} \ol{f}-\ol{g} {\ol{f}}^{3}+6 f \ol{f} g \ol{g}-f^{3} g -f \,g^{3}-f \ol{g}-g \ol{g}-f \ol{f}-\ol{f} g +1\right) q^{4} p^{4}\\
+&\left({\ol{g}}^{3} \ol{f} f g -2 {\ol{g}}^{2} {\ol{f}}^{2} f g +f g \ol{g} {\ol{f}}^{3}+\ol{g} \ol{f} f^{3} g -2 \ol{g} \ol{f} f^{2} g^{2}+\ol{g} \ol{f} f \,g^{3}+{\ol{g}}^{2} \ol{f} f +{\ol{g}}^{2} \ol{f} g +\ol{g} {\ol{f}}^{2} f\right.\\
+&\left.\ol{g} {\ol{f}}^{2} g +\ol{g} f^{2} g +\ol{g} f \,g^{2}+\ol{f} f^{2} g +\ol{f} f \,g^{2}-2 \ol{f} \ol{g}-2 f g \right) p^{3} q^{5}\\
+&\left(2 {\ol{g}}^{2} {\ol{f}}^{2} f^{2} g^{2}-{\ol{g}}^{2} \ol{f} f^{2} g -{\ol{g}}^{2} \ol{f} f \,g^{2}-\ol{g} {\ol{f}}^{2} f^{2} g -\ol{g} {\ol{f}}^{2} f \,g^{2}+{\ol{f}}^{2} {\ol{g}}^{2}+4 f \ol{f} g \ol{g}+f^{2} g^{2}\right) p^{2} q^{6}\\
-&2 f g \,q^{7}  
\ol{f} \ol{g} \left(\ol{f} \ol{g}+f g \right) p +f^{2} g^{2} q^{8} {\ol{f}}^{2} {\ol{g}}^{2}\\
\gamma_2=&4 \left(f \ol{f} g \ol{g} p^{4}+\left(-\ol{f} \ol{g}-f g \right) q \,p^{3}+q^{2} \left(f \ol{f} g \ol{g}+1\right) p^{2}+\left(-\ol{f} \ol{g}-f g \right) p \,q^{3}+f \ol{f} g \ol{g} q^{4}\right)^{2}
\end{align*}}


\bibliographystyle{maa}
\bibliography{refs,refs_00_book,refs_01_pub,refs_02_acc,refs_04_unsub}

\end{document}